\documentclass[10pt,a4paper]{article}
\usepackage{amssymb}
\usepackage{graphicx}
\usepackage{caption}
\usepackage{amsmath}
\usepackage{xcolor}
\newtheorem{theorem}{Theorem}

\newtheorem{definition}{Definition}

\newtheorem{remark}{Remark}

\newenvironment{proof}[1][Proof]{\textbf{#1.} } {\ \ \hfill\hbox to .1pt{} \hfill\hbox to .1pt{}
\hfill$\blacksquare$\par}
\begin{document}
\centerline{\Large{\bf{Observer design and practical  stability }}}
\centerline{\Large{\bf{of nonlinear systems under unknown time-delay}}}
\centerline{}
\centerline{\bf {Nadhem  ECHI}}
\centerline{}

\centerline{Gafsa University, Faculty of Sciences of Gafsa}
\centerline{Department of Mathematics, Zarroug   Gafsa 2112 Tunisia}
\centerline{E-mail: nadhemechi\_fsg@yahoo.fr}

\abstract{In the present  paper, we study  observer design and
 we establish  some sufficient conditions for
  practical exponential stability for a class of time-delay nonlinear systems written in triangular form.
In case of delay, the exponential convergence of the observer was confirmed.
   Based on the Lyapunov-Krasovskii functionals, the practical stability of the proposed observer is achieved. Finally,
   a physical model and
simulation findings show  the feasibility of the suggested strategy.
}
 \bigskip

 \vspace*{0.1cm}

\vspace*{0.1cm}
\section{Introduction}
Time-delay systems are one of the  basic mathematical models of real phenomena such as nuclear reactors,
 chemical engineering systems, biological systems
\cite{lili}, and population dynamics models \cite{muroya}.
The analysis of systems without delays is generally simple as compared to nonlinear systems  under time delays. However,
 there are a number of  problems relating to nonlinear observer for time delay system.
In particular, a problem of theoretical and practical importance is the design of observer-based  for time-delay systems.
In literature, a separation principle, for nonlinear free-delay systems, using high gain observers is provided in \cite{Atassi99} and \cite{Atassi00}.
Much attention has been paid to solve such a problem and many observer designs for time-delay nonlinear systems  approaches have been used.
In \cite{zhou}, it is shown that observer design for a class of nonlinear time delay systems is solved by using   linear matrix inequality.
 In \cite{hamed2009}, some sufficient conditions for
 practical uniform  stability of a class of uncertain time-varying systems with a bounded time-varying state delay
 were provided using the  Lyapunov stability theory.
  \cite{nadhem} and \cite{nadhem2017} show that state and output feedback controllers of time-delay systems,
written in a triangular  linear growth condition are reached under delay independent conditions and under delay dependent conditions respectively.

The observer design problem for nonlinear systems satisfying a Lipschitz continuity condition has been a topic of  numerous papers,
such as for nonlinear free-delay systems \cite{Atassi99, Atassi00, Zhu, Rajamani}, for nonlinear systems with
unknown, time-varying \cite{omar,ghanes, Farza}.
A reduced-order observer design method is presented in \cite{Zhu} for a class of
Lipschitz nonlinear continuous-time systems  without time delays which extend the results in \cite{Rajamani}.

However, in practice, dynamics, measurement, noises or disturbances
often prevent the error signals from tending to the origin.
Thus, the origin is not  a point of equilibrium of the system. An
  additive term on the right-hand side of the nonlinear system is used to present the uncertainties systems. For this reason,
the property is referred to as practical stability which is more suitable  for nonlinear free-delay systems ( see \cite{amel2009, corles})
and for nonlinear systems with  time-delay  ( see \cite{ hamed, ghanes, omar, raul}).
Under unknown, bounded  time-delay, an observer design for a class of nonlinear system is presented in \cite{omar}.
\cite{hamed} concluded that a class  of nonlinear time delay systems is conformed due to some assumptions and the
 time varying delay bounded the practical
exponential stability.
Based on conditions in terms of Ricatti differential equation, the problem of the  practical exponential stability of a
class of delayed nonlinear systems is proved in  \cite{omar2015}.
 \cite{omar:2015} investigated the problem of design for a class of nonlinear systems under unknown time-varying delay. Based on sufficient assumptions,
the practical and the exponential stability is achieved.

The main aim of the current  paper is to generalize the idea investigated, for instance, in \cite{ ghanes, omar}
for the purpose of establishing   the design of observer. We investigate the problem of   exponential convergence of the observation error
 of a class of nonlinear time-delay systems with constant delay. We impose a generalized condition on the nonlinearity to cover the time-delay systems
considered in \cite{Be} and a class of systems considered in \cite{ghanes} and \cite{omar} .
Under unknown and variable time delay and by constructing Lyapunov–Krasovskii functionals,
new criteria are given to insure the practical stability in which the error converges to a small ball.

The rest of this paper is organized as follows: In Section 2, the exponential stable and the practical stability definition are presented
 and the system description is given. The observer design synthesis method and its stability analysis for a class of nonlinear systems are proved
  in section 3. In section 4, we illustrate our results by a physical model.

\section{System description and basic results}
  Consider  time delay system of the form:
\begin{equation}\left\{
                   \begin{array}{ll}
                     \dot{x}(t)=f(t,x(t),x(t-\tau(t))) & \hbox{} \\
                     x(s)=\varphi(s) & \hbox{ }\\
                   \end{array}
                 \right.
\label{1}
 \end{equation}
where $\tau(t)$ represents a positive real-value unknown
function that denotes the time varying delay affecting both
state and input of the system, $x(t, \varphi)$ is the solution of the system with initial function
$\varphi$, verifying:
$$x(s, \varphi) = \varphi(s),\ \ \ \ \forall s\in[-\tau; 0].$$
$ \varphi\in\mathcal{C}$ where $\mathcal{C}$ denotes the Banach space
of continuous functions mapping the interval $[-\tau, 0]\rightarrow\mathbb{R}^{n}$ equipped with the supremum-norm:
$$\parallel \varphi\parallel_{\infty}\ =\ \max_{s\in[-\tau,0]}\parallel\varphi(s)\parallel$$
$\| \ \|$ being the Euclidean-norm. The map $f :  \mathbb{R}_{+}\times\mathbb{R}^{n} \times \mathbb{R}^{n}\rightarrow \mathbb{R}^{n}$
is a piecewise continuous function in $t$, and locally Lipschitz in $x$, and satisfies
$f(t,0, 0) = 0,\ \forall t\geq0.$

For $r > 0$, denote $B_{r} = \{x \in \mathbb{R}^{n}/\|x\| \leq r\}.$
In the case when $f(t, 0,0)\neq 0$, for certain $t \geq 0$, we shall study the problem of asymptotic
stability not for the origin but for a neighborhood of the origin approximated by a small ball of radius $r > 0$ centrad at the origin.

The function segment $x_{t}$ is defined by $x_{t}(\theta) = x(t + \theta),\,\ \theta \in[-\tau, 0].$
For $\varphi \in \mathcal{C}$, we
denote by $x(t,\varphi)$ or shortly $x(t)$ the solution of \eqref{1} that satisfies $x_{0} =\varphi. $ The segment
of this solution is denoted by $x_{t}(\varphi)$ or shortly $x_{t}$.
 \begin{definition}\cite{pepek}
  The zero solution of system \eqref{1} is said to be  globally exponentially stable with
a decay rate $\alpha>0$, if there exist positive reals $\alpha $ and $\beta$ such
that, for all $t\geq t_{0}$ and $\varphi\in \mathcal{C}$, the following inequality holds:
\[ \parallel x(t)\parallel\leq \beta \|\varphi\|_{\infty}\exp{(-\alpha (t-t_{0}))}.\]
    \end{definition}
 \begin{definition}\cite{hamed}
We say that $B_{r}$ is globally uniformly exponentially
stable if there exist $\lambda_{1} >0$ and $\lambda_{2}> 0$ such that for
all $t\geq t_{0}$  and $\varphi \in \mathcal{C}$, we have
$$\|x(t)\|\leq r+\lambda_{1}\|\varphi\|_{\infty}\exp(-\lambda_{2}(t-t_{0})).$$

System \eqref{1} is globally uniformly practically exponentially stable if there exists $r > 0$ such that $B_{r}$
is globally uniformly exponentially stable. \end{definition}
\begin{remark}
When $r = 0$, in this case the origin is an equilibrium point, then we
point the classical definition of the exponential stability (see\cite{nadhem},\cite{pepek}).\end{remark}
\begin{remark}
The global uniform practical asymptotic stability of a ball $B_{r}$
defined in this paper is less restrictive than the stability of compact set given
in \cite{lin1996} of free-delay systems.\end{remark}
 In this paper, we consider the time delay nonlinear system
  \begin{equation}\left\{
                   \begin{array}{ll}
                     \dot{x}(t)=Ax(t)+f(x(t),x(t-\tau(t)),u(t),u(t-\tau(t))),\ \ \ \ \ t\geq0 & \hbox{,} \\
                     y(t)=Cx(t) & \hbox{,}\\
                     x(s)=\varphi(s),\ \ \ \ \ \ \ \ \ \ \ \ \ \ \ \ \ \ \  \ \ \ \ \ \  \ \ \ \
                     \ \ \ \ \ \ \ \ \ \ \ \ \ \ \ \ \ \ \forall s\in[-\tau,0] & \hbox{. }
                   \end{array}
                 \right.
 \label{3}\end{equation}
where $ x(t)\in\mathbb{R}^{n}$ is the state vector, $u(t)\in \mathbb{R}^{m}$ is the
 input of the system, $y(t)\in \mathbb{R}$
is the measured output, $\tau (t)$ is a continuously differentiable function which
denotes the time-varying delay, $x(t-\tau(t))$ and $u(t-\tau(t))$ are,
 respectively, the delayed state and input. The matrices $A,$ and $C$ are given by,
 $$A=\left[
 \begin{array}{ccccc}
  0 & 1 &0& \cdots & 0 \\
  0 & 0 & 1 & \cdots&0 \\
 \vdots & \vdots& \vdots & \ddots & \vdots \\
 0 & 0 & 0&\cdots & 1 \\
 0 & 0 & 0&\cdots & 0 \\
 \end{array}
 \right]
,\, C=\left[
    \begin{array}{ccccc}
      1 & 0 & \cdots&0 & 0 \\
    \end{array}
  \right]$$
  and the perturbed term is
$$f(x(t), x(t -\tau(t) , u(t),u(t-\tau(t))) =
 [f_{1}(x(t), x(t -\tau(t) ),u(t),u(t-\tau(t))), \cdots , f_{n}(x(t), x(t -\tau(t) ),u(t),u(t-\tau(t)))]^{T}.$$
The mappings
 $f_{i}:\mathbb{R}^{n}\times\mathbb{R}^{n}\times\mathbb{R}^{m}\times\mathbb{R}^{m}\rightarrow\mathbb{R},$ $i=1,\ldots,n,$ are smooth with
 $f_{i}(0,0,0,0)=0.$

 Throughout the paper, the time argument is omitted and the delayed state vector $x(t - \tau(t))$ is noted by $x^{\tau(t)}$.  $A^{T}$ means the
transpose of $A$.
$\lambda_{max}(A)$ and $\lambda_{min}(A)$ denote the maximal and minimal eigenvalue of a matrix
 $A$ respectively.
\section{ Main results}
We suppose that $f$ satisfies the following assumption:

 $\textbf{A1.}$  There exists functions
$\gamma_{1}(\varepsilon)>0$, $\gamma_{2}(\varepsilon)>0$  and $\gamma_{3}(\varepsilon)>0$ such that for $\varepsilon>0$,
\begin{eqnarray}\displaystyle\sum_{i=1}^{n}\varepsilon^{i-1}|f_{i}(x,\overline{x},u,u^{\tau(t)})
-f_{i}(y,\overline{y},u,u^{\tau(t)})|&\leq&\gamma_{1}(\varepsilon)
\displaystyle\sum_{i=1}^{n}\varepsilon^{i-1}|x_{i}-y_{i}|+\gamma_{2}(\varepsilon)
\displaystyle\sum_{i=1}^{n}\varepsilon^{i-1}|\overline{x}_{i}-\overline{y}_{i}|\label{A1},\\ \nonumber\\
\displaystyle\sum_{i=1}^{n}\varepsilon^{i-1}|f_{i}(x,\overline{x},u,u^{\tau(t)})
-f_{i}(x,\overline{y},u,\overline{u}^{\tau(t)})|&\leq& \gamma_{2}(\varepsilon)
\displaystyle\sum_{i=1}^{n}\varepsilon^{i-1}|\overline{x}_{i}-\overline{y}_{i}|
+\gamma_{3}(\varepsilon)
\displaystyle\sum_{i=1}^{n}\varepsilon^{i-1}\left(\|u^{\tau(t)}-\overline{u}^{\tau(t)}\|\right).\end{eqnarray}
\begin{remark}\label{rem1}
We can easily show if the system  \eqref{3} has a triangular structure  (see \cite{ibrir},\cite{ghanes}), that is
each $f_{i}$ depends only on $(x_{1},\cdots, x_{i},x_{1}^{\tau},\cdots, x_{i}^{\tau},u,u^{\tau(t)} )$ and if we suppose that $f_{i}$ is globally
Lipschitz with
respect to $(x_{1},\cdots, x_{i}),\ (x_{1}^{\tau},\cdots, x_{i}^{\tau})$ and $u^{\tau(t)}$,
uniformly with respect to $u$, which implies that
\begin{eqnarray} |f_{i}(x,\overline{x},u,u^{\tau(t)})
-f_{i}(y,\overline{y},u,\overline{u}^{\tau(t)})|&\leq& |f_{i}(x,\overline{x},u,u^{\tau(t)})
-f_{i}(x,\overline{z},u,\overline{u}^{\tau(t)})|+|f_{i}(x,\overline{z},u,\overline{u}^{\tau(t)})
-f_{i}(y,\overline{y},u,\overline{u}^{\tau(t)})|\nonumber\\ \nonumber\\
 &\leq& k_{1}\displaystyle\sum_{j=1}^{i}(| \overline{x}_{j}-\overline{z}_{j}|)+k_{1}\|u^{\tau(t)}-\overline{u}^{\tau(t)}\|\label{f1}\\
& &+ k_{2}\displaystyle\sum_{j=1}^{i}(| x_{j}-y_{j}|+| \overline{z}_{j}-\overline{y}_{j}|),
\label{lips}\end{eqnarray}
where $k_{1} > 0, \ (k_{2} > 0)$ is a Lipschitz constant in \eqref{f1}, and  in \eqref{lips} respectively,
then assumption $\textbf{A1}$ is fulfilled. Indeed,
$$\begin{array}{lll}  \displaystyle\sum_{i=1}^{n}\varepsilon^{i-1}|f_{i}(x,\overline{x},u,u^{\tau(t)})
-f_{i}(y,\overline{y},u,\overline{u}^{\tau(t)})|&\leq &\displaystyle\sum_{i=1}^{n}\varepsilon^{i-1}
k_{1}\displaystyle\sum_{j=1}^{i}(| \overline{x}_{j}-\overline{z}_{j}|)+\displaystyle\sum_{i=1}^{n}\varepsilon^{i-1}
k_{1}\|u^{\tau(t)}-\overline{u}^{\tau(t)}\|\\\\
& &+ \displaystyle\sum_{i=1}^{n}\varepsilon^{i-1}
k_{2}\displaystyle\sum_{j=1}^{i}(| x_{j}-y_{j}|+| \overline{z}_{j}-\overline{y}_{j}|)
\end{array}$$
But, on the one hand,  we have
$$\begin{array}{lll}  \displaystyle\sum_{i=1}^{n}\varepsilon^{i-1}
k_{2}\displaystyle\sum_{j=1}^{i}(|x_{j}-y_{j}|&=&k_{2}(1+\varepsilon+\cdots+\varepsilon^{n-1})|x_{1}-y_{1}|\\
& &+k_{2}(\varepsilon+\cdots+\varepsilon^{n-1})|x_{2}-y_{2}|+\cdots+ k_{2}\varepsilon^{n-1}|x_{n}-y_{n}|\\\\
&=&k_{2}(1+\varepsilon+\cdots+\varepsilon^{n-1})|x_{1}-y_{1}|\\
& &+k_{2}\varepsilon(1+\cdots+\varepsilon^{n-2})|x_{2}-y_{2}|+\cdots+ k_{2}\varepsilon^{n-1}|x_{n}-y_{n}|\\\\
&\leq&k_{2}(1+\varepsilon+\cdots+\varepsilon^{n-1})\displaystyle\sum_{i=1}^{n}\varepsilon^{i-1}|x_{i}-y_{i}|,
\end{array}$$
and thus also
$$\begin{array}{lll}  \displaystyle\sum_{i=1}^{n}\varepsilon^{i-1}
k_{2}\displaystyle\sum_{j=1}^{i}(|\overline{z}_{j}-\overline{y}_{j}|
&\leq&k_{2}(1+\varepsilon+\cdots+\varepsilon^{n-1})\displaystyle\sum_{i=1}^{n}\varepsilon^{i-1}|\overline{x}_{i}-\overline{y}_{i}|,\\\\
 \displaystyle\sum_{i=1}^{n}\varepsilon^{i-1}
k_{1}\displaystyle\sum_{j=1}^{i}(|\overline{x}_{j}-\overline{y}_{j}|
&\leq&k_{1}(1+\varepsilon+\cdots+\varepsilon^{n-1})\displaystyle\sum_{i=1}^{n}\varepsilon^{i-1}|\overline{x}_{i}-\overline{y}_{i}|.
\end{array}$$
and on the other hand, we have
$$\begin{array}{lll}\displaystyle\sum_{i=1}^{n}\varepsilon^{i-1}
k_{1}\|u^{\tau(t)}-\overline{u}^{\tau(t)}\|
&\leq&k_{1}(1+\varepsilon+\cdots+\varepsilon^{n-1})\displaystyle\sum_{i=1}^{n}\varepsilon^{i-1}
\|u^{\tau(t)}-\overline{u}^{\tau(t)}\|.
\end{array}$$
So assumption $\textbf{A1.} $ is satisfied with
$$\gamma_{1}(\varepsilon) = \gamma_{2}(\varepsilon) = \gamma_{3}(\varepsilon) = k(1 +\varepsilon + \cdots + \varepsilon^{n-1}),$$
where $k=\max(k_{1},\, k_{2})$.
\end{remark}
\begin{remark}
In the paper, we deals with the more general systems where the nonlinear function is not necessarily Lipschitz.
\end{remark}
\subsection{Observer design}
A tenth of researchers has studied the  design problem of the observer for example in
 \cite{germani}, \cite{ibrir}, \cite{omar} and references
therein. For nonlinear systems having triangular structures, the global asymptotic stability
is proved using a high-gain parameterized
linear controller in \cite{ibrir}.
Under some condition, \cite{germani} investigated the problem for exponential observation for nonlinear delay systems.
 \cite{omar} presented
observer design for a class of nonlinear system with bounded time-varying delay.

To complete the description of system \eqref{3}, the following assumption is considered.\\
$\textbf{A2} $ For $t \geq 0,$ the time delay $\tau(t)=\tau$ is known and constant.

In this subsection, under constant and known time delay, we  present delay-independent conditions to ensure exponential
 convergence of the observation error. To define the nonlinear time-delay observer for system \eqref{3}
  under assumptions $\textbf{A1}$ and $\textbf{A2}$, The following state observer is proposed:
\begin{equation}
\left\{
  \begin{array}{ll}
    \dot{\hat{x}}(t) = A\hat{x}(t) +f(\hat{x}(t),\hat{x}^{\tau},u(t),u^{\tau})  + L(\varepsilon)(C\hat{x}(t) -y(t)), & \hbox{} \\
    \hat{y}(t) = C\hat{x}(t) , & \hbox{}
  \end{array}
\right.
\label{chpx}\end{equation}
where $L(\varepsilon)= [\frac{l_{1}}{\varepsilon},\ldots,
\frac{l_{n}}{\varepsilon^{n}}]^{T}$ and $L= [l_{1},\ldots,l_{n}]^{T}$ such that
$A_{L}:=A+LC$ is Hurwitz. Let $P$ be the symmetric positive definite solution of the Lyapunov equation
\begin{equation} A^{T}_{L}P + PA_{L} = -I\label{Ly2}\end{equation}

\begin{theorem}\label{theorem1}
Consider the time-delay system \eqref{3} under assumptions $\textbf{A1.}$ and $\textbf{A2.}$. Suppose
that there exists $\varepsilon > 0$ such that
\begin{equation}\frac{\lambda_{\min}(P)}{\varepsilon\|P\|}
-2 n\gamma_{1}(\varepsilon) \| P\| -n^{2}\gamma_{2}^{2}(\varepsilon)\| P\|^{2}-1>0\label{condi1}\end{equation}
Then, system \eqref{chpx} is a globally  exponential observer
for system \eqref{3}.
\end{theorem}
\begin{proof}
Denote $e=\hat{x}-x$ the observation error.
 We have \begin{equation}\dot{e}=(A+L(\varepsilon)C)e+f(\hat{x},\hat{x}^{\tau},u,u^{\tau})-f(x,x^{\tau},u,u^{\tau})\label{erro}\end{equation}
 For $\varepsilon>0$, let $D(\varepsilon)=diag[1,\varepsilon,\ldots,\varepsilon^{n-1}]$.
Let $\eta=D(\varepsilon)e$. Using the fact that $A+L(\varepsilon)C=
\frac{1}{\varepsilon}D(\varepsilon)^{-1}A_{L}D(\varepsilon),$ we get
\begin{equation}\dot{\eta}=\frac{1}{\varepsilon}A_{L}\eta+D(\varepsilon)(f(\hat{x},\hat{x}^{\tau},u,u^{\tau})-f(x,x^{\tau}u,u^{\tau}))\label{eta}
 \end{equation}
Let us choose a Lyapunov-Krasovskii functional candidate as follows
\begin{equation} V(\eta_{t})= V_{1}(\eta_{t})+ V_{2}(\eta_{t})\label{veta}\end{equation}
with $$V_{1}(\eta_{t})=\eta^{T} P\eta$$ and
$$ V_{2}(\eta_{t}) = \displaystyle\int_{t-\tau}^{t}e^{\frac{\sigma}{\tau}( s-t)}\|\eta(s)\|^{2}ds$$
with $\sigma$ a positive constant defined thereafter.\\
The time derivative of $V_{1}(\eta_{t} )$ along the trajectories of
system \eqref{eta} is
$$\dot{V}_{1} (\eta_{t})=\frac{1}{\varepsilon}  \eta^{T} ( A^{T}_{L}P + PA_{L} )\eta+
2\eta^{T} PD(\varepsilon)(f(\hat{x},\hat{x}^{\tau},u,u^{\tau})-f(x,x^{\tau}u,u^{\tau}))$$
The time derivative of $V_{2}(\eta_{t} )$ along the trajectories of
system \eqref{eta} is
$$\dot{V}_{2} (\eta_{t})=\| \eta\|^{2}
-e^{-\sigma }\|\eta^{\tau}\|^{2}
-\frac{\sigma}{ \tau}\displaystyle\int_{t-\tau}^{t}  e^{\frac{\sigma }{\tau}(s-t)}\|\eta(s)\|^{2}ds$$
Since $P$ is symmetric positive definite, for all $\eta\in\mathbb{R}^{n},$
\begin{equation}\lambda_{\min}(P)\|\eta\|^{2}\leq \eta^{T}P\eta\leq \lambda_{\max}(P)\|\eta\|^{2}\label{norm}\end{equation}
Taking the time derivative of \eqref{veta} along the trajectories of \eqref{eta}, and making use of \eqref{Ly2} and \eqref{norm}, we have
$$\begin{array}{lll}
\dot{V} (\eta_{t})
&\leq& -\frac{1}{\varepsilon\|P\|}V_{1}(\eta)+
2\|\eta \|\| P\|\| D(\varepsilon)(f(\hat{x},\hat{x}^{\tau},u,u^{\tau})-f(x,x^{\tau}u,u^{\tau}))\|\\\\
& &+\|\eta \|^{2}-e^{-\sigma}\|\eta^{\tau}\|^{2}
-\frac{\sigma}{ \tau}V_{2}(\eta).\end{array}$$
Now \begin{equation}\|D(\varepsilon)(f(\hat{x},\hat{x}^{\tau},u,u^{\tau})-f(x,x^{\tau}u,u^{\tau}))\|
\leq \displaystyle\sum_{i=1}^{n}\varepsilon^{i-1}|f_{i}(\hat{x},\hat{x}^{\tau},u,u^{\tau})-f_{i}(x,x^{\tau},u,u^{\tau})|.\label{dd}\end{equation}
So using assumption $\textbf{A1.}$ equation\eqref{A1}, we get $$\begin{array}{lll}
\|D(\varepsilon)(f(\hat{x},\hat{x}^{\tau},u,u^{\tau})-f(x,x^{\tau}u,u^{\tau}))\|&\leq&
\gamma_{1}(\varepsilon)\displaystyle\sum_{i=1}^{n}\varepsilon^{i-1}|\hat{x}_{i}-x_{i}|+\gamma_{2}(\varepsilon)
\displaystyle\sum_{i=1}^{n}\varepsilon^{i-1}|\hat{x}_{i}^{\tau}-x_{i}^{\tau}|\\
&\leq &n\gamma_{1}(\varepsilon) \| D(\varepsilon)e\|
+n\gamma_{2}(\varepsilon) \| D(\varepsilon)e^{\tau}\|.\end{array}$$
Thus \begin{equation}\|D(\varepsilon)(f(\hat{x},\hat{x}^{\tau},u,u^{\tau})-f(x,x^{\tau}u,u^{\tau}))\| \leq
n\gamma_{1}(\varepsilon) \| \eta\| +n\gamma_{2}(\varepsilon) \| \eta^{\tau}\| \label{7}
\end{equation}
Thus, we have that $$\begin{array}{lll}
\dot{V} (\eta_{t})+\frac{\sigma}{\tau} V(\eta_{t})&\leq& -(\frac{1}{\varepsilon\|P\|}-\frac{\sigma}{\tau})V_{1}(\eta)
+2 n\gamma_{1}(\varepsilon) \| P\|  \| \eta\|^{2}\\\\
& &+ 2n\gamma_{2}(\varepsilon)
\| P\|  \| \eta \|  \| \eta^{\tau}\|+\|\eta\|^{2}- e^{-\sigma }\|\eta^{\tau}\|^{2}\\\\
&\leq& -\left\{\lambda_{\min}(P)(\frac{1}{\varepsilon\|P\|}-\frac{\sigma}{\tau})-2 n\gamma_{1}(\varepsilon) \| P\| -1\right\}\| \eta\|^{2}\\\\
& &+2n\gamma_{2}(\varepsilon)\| P\|  \| \eta \|  \| \eta^{\tau}\|- e^{-\sigma }\|\eta^{\tau}\|^{2}
\end{array}$$
Hence, we have that\begin{equation}\dot{V} (\eta_{t})+\frac{\sigma}{\tau} V(\eta_{t})\leq -a(\sigma,\varepsilon) \|\eta\|^{2}+
b(\varepsilon)\|\eta \|  \| \eta^{\tau}\|- e^{-\sigma }\|\eta^{\tau}\|^{2}\label{v+v}\end{equation}
with
$$\begin{array}{lll}a(\sigma,\varepsilon)&=&\lambda_{\min}(P)(\frac{1}{\varepsilon\|P\|}-\frac{\sigma}{\tau})-2 n\gamma_{1}(\varepsilon) \| P\| -1,\\\\
\ \ b(\varepsilon)&=&2n\gamma_{2}(\varepsilon)\| P\|. \end{array}$$
Now, the right side of the  inequality \eqref{v+v} can be rewritten as follows
$$-a(\sigma,\varepsilon) \|\eta\|^{2}+
b(\varepsilon)\|\eta \|  \| \eta^{\tau}\|- e^{-\sigma }\|\eta^{\tau}\|^{2}=-\left(a(\sigma,\varepsilon)-\frac{b^{2}(\varepsilon)}{4}e^{\sigma}\right)
\|\eta\|^{2}-\left(\frac{b(\varepsilon)}{2}e^{\frac{\sigma}{2} }\|\eta\|-e^{-\frac{\sigma}{2} }\|\eta^{\tau}\|\right)^{2}$$
To satisfy inequality \eqref{v+v}, all we need to do is to choose $\sigma$ such that
$$\left(a(\sigma,\varepsilon)-\frac{b^{2}(\varepsilon)}{4}e^{\sigma}\right)>0,$$
which is equivalent to
\begin{equation} \lambda_{\min}(P)\frac{\sigma}{\tau}+\frac{b^{2}(\varepsilon)}{4}(e^{\sigma}-1)<\frac{\lambda_{\min}(P)}{\varepsilon\|P\|}
-2 n\gamma_{1}(\varepsilon) \| P\| -1-\frac{b^{2}(\varepsilon)}{4}\label{sigma}\end{equation}

Let $w(x)=\lambda_{\min}(P)\frac{x}{\tau}+\frac{b^{2}(\varepsilon)}{4}(e^{x}-1)$, we have $w(x) > 0\ \forall x > 0$ and $w(0) = 0$.\\
Since $w$ is continuous at $0$, there exists $\delta > 0$ such that $\forall x\in]0, \delta[,
0 < w(x) <\frac{\lambda_{\min}(P)}{\varepsilon\|P\|}
-2 n\gamma_{1}(\varepsilon) \| P\| -\frac{4+b^{2}(\varepsilon)}{4}.$
Let $\sigma \in]0, \delta[,$ then inequality \eqref{sigma} is verified.

Now, the objective is to prove the exponential convergence of the observer \eqref{erro}.\\
Using \eqref{condi1}, inequality \eqref{v+v} becomes
$$\dot{V}(\eta_{t})\leq-\frac{\sigma}{\tau} V(\eta_{t})$$
It follows that\begin{equation}V(\eta_{t})\leq e^{-\frac{\sigma}{\tau}t} V(\eta_{t}(0))\label{gronw}\end{equation}
Using \eqref{veta} and \eqref{norm}, we have, on the one hand,
$$\begin{array}{lll}
V (\eta_{t}(0))&\leq& \lambda_{\max}(P)\|\eta_{t}(0)\|^{2}+\displaystyle\int_{-\tau}^{0}e^{\frac{\sigma}{\tau}s}\|\eta(s)\|^{2}ds\\\\
&\leq &(\lambda_{\max}(P)+\tau)\displaystyle\sup_{s\in[-\tau,0]}\|\eta(s)\|^{2}
\end{array}$$
and on the other hand,$$\lambda_{\min}(P)\|\eta_{t}\|^{2}\leq V (\eta_{t}).$$
We deduce that $$\|\eta_{t}\|\leq \sqrt{\frac{\|P\|+\tau}{\lambda_{\min}(P)}}e^{\frac{-\sigma}{2\tau}t}\displaystyle\sup_{s\in[-\tau,0]}\|\eta(s)\|.$$
Finally, with $\eta = D(\varepsilon)e$, the observation error $e(t)$ is given by
$$\|e(t)\|\leq \frac{1}{\|D(\varepsilon)\|}\sqrt{\frac{\|P\|+\tau}{\lambda_{\min}(P)}}
e^{\frac{-\sigma}{2\tau}t}\displaystyle\sup_{s\in[-\tau,0]}\|\eta(s)\|.$$
Then, the error dynamics \eqref{erro} is globally exponentially stable.\end{proof}
\begin{remark}
In \cite{zhou}, based on linear matrix inequalities, the authors developed the sufficient conditions which guarantee
 the estimation error converge asymptotically towards zero.
As compared to   \cite{zhou}, our results are less conservative and more convenient to use since they are independent of time delays.
\end{remark}
\begin{remark}
 \cite{Dong:2012}, proposed a state feedback controller that are synthesized under sufficient conditions expressed in terms of Riccati differential
 equations and linear matrix inequalities which can stabilize the studied nonlinear uncertain systems with time-varying delay.
  Feedback controllers are synthesized under sufficient conditions linear matrix inequalities and expressed in terms of
 Riccati differential equations. But, in this paper, we use a parameter in order to establish global asymptotical stability of the nonlinear system.
\end{remark}
\subsection{Practical exponential stability}
In this section, we give sufficient conditions to ensure the  practical stability convergence.
In fact, in the real world, the problem of practical stability is more appropriate.
 Then, for practical purpose,
 practical stability seems desirable ( see \cite{amel2009} ) for systems without delays
and \cite{raul},\cite{hamed} and \cite{ghanes} for time-delay systems.
In the general case, one can not directly measure the states of a system.
Thus, one must observe the unmeasured states.
An observer is a dynamical system which estimates
the states of the system.

In this section, for complete the description of system \eqref{3} the following assumptions are needed.

$\textbf{H1.}$ The state and the input are considered bounded, that
is $x(t)\in K\subset\mathbb{R}^{n}$( that is a compact subset of $\mathbb{R}^{n})$.

$\textbf{H2.}$ The time-varying delay satisfies the following properties:
\begin{description}
  \item[(i)]
 $\exists\  \tau^{*} > 0$, such that $0\leq\tau(t) \leq\tau^{*}$.
\item[(ii)] $\exists\ \beta > 0$, such that $ \dot{\tau}(t) \leq1- \beta$.\end{description}
\begin{remark}The boundedness of the state excludes implicitly all initial conditions that
generate unbounded state.
\end{remark}

The following state observer for system \eqref{3} under assumption $\textbf{H1.}, \textbf{H2.}$ and $\textbf{A1.}$ is proposed:
\begin{equation}
\left\{
  \begin{array}{ll}
    \dot{\hat{x}}(t) = A\hat{x}(t) +f(\hat{x}(t),\hat{x}^{\tau^{*}},u(t),u^{\tau^{*}})  + L(\varepsilon)(C\hat{x}(t) -y(t)), & \hbox{} \\
    \hat{y}(t) = C\hat{x}(t) , & \hbox{}
  \end{array}
\right.\label{chapx2}\end{equation}
Let us now define $e=\hat{x}-x$ the observation error, which
denotes the difference between the actual state and
estimated states.
\begin{theorem}\label{th2}
Consider the time-delay system \eqref{3} under assumptions $\textbf{H1.}, \textbf{H2.}$ and $\textbf{A1.}$. Suppose
that there exists $\varepsilon > 0$ such that
\begin{equation}\frac{\lambda_{\min}(P)}{\varepsilon\|P\|}
-2 n\gamma_{1}(\varepsilon) \| P\|-\frac{5}{4} -n^{2}\gamma_{2}^{2}(\varepsilon)\| P\|^{2}>0\label{condi2}\end{equation}
Then, the error dynamics \eqref{erro2} is globally (on $K$) practically exponentially
stable.
\end{theorem}
\begin{proof}
 We have
\begin{equation}\dot{e}=(A+L(\varepsilon)C)e+f(\hat{x},\hat{x}^{\tau^{*}},u,u^{\tau^{*}})-f(x,x^{\tau(t)},u,u^{\tau(t)})\label{erro2}\end{equation}
 For $\varepsilon>0$, let $D(\varepsilon)=diag[1,\varepsilon,\ldots,\varepsilon^{n-1}]$.
Let $\eta=D(\varepsilon)e$. Using the fact that $A+L(\varepsilon)C=
\frac{1}{\varepsilon}D(\varepsilon)^{-1}A_{L}D(\varepsilon),$ we get
\begin{equation}
\dot{\eta}=\frac{1}{\varepsilon}A_{L}\eta+D(\varepsilon)(f(\hat{x},\hat{x}^{\tau^{*}},u,u^{\tau^{*}})-f(x,x^{\tau(t)},u,u^{\tau(t)}))\label{eta2}
\end{equation}
Let us choose a Lyapunov-Krasovskii functional candidate as follows
\begin{equation}
 W(t,\eta_{t})= \eta^{T} P\eta+ \displaystyle\int_{t-\tau(t)}^{t}e^{\frac{\sigma}{\tau^{*}}( s-t)}\|\eta(s)\|^{2}ds\label{veta2}\end{equation}
where $P$ is provided by \eqref{Ly2} and $\sigma$ a positive constant defined thereafter.\\
The time derivative of $W(t,\eta_{t} )$ along the trajectories of
system \eqref{eta2} is
$$\begin{array}{lll}\dot{W}(t,\eta_{t})&=&\frac{1}{\varepsilon}  \eta^{T} ( A^{T}_{L}P + PA_{L} )\eta+
2\eta^{T} PD(\varepsilon)(f(\hat{x},\hat{x}^{\tau^{*}},u,u^{\tau^{*}})-f(x,x^{\tau(t)},u,u^{\tau(t)}))\\\\
& &+\| \eta\|^{2} -(1-\dot{\tau}(t))e^{-\frac{\sigma \tau(t)}{\tau^{*}} }\|\eta^{\tau}\|^{2}
-\frac{\sigma}{ \tau}\displaystyle\int_{t-\tau(t)}^{t}  e^{\frac{\sigma }{\tau^{*}}(s-t)}\|\eta(s)\|^{2}ds.\end{array}$$
As in the proof of Theorem \ref{theorem1} and using \eqref{Ly2}, \eqref{norm} and assumption $\textbf{H2.}$, we have
$$\begin{array}{lll}
\dot{W} (t,\eta_{t})+\frac{\sigma}{\tau^{*}}W (t,\eta_{t})
&\leq& -\left\{\lambda_{\min}(P)(\frac{1}{\varepsilon\|P\|}-\frac{\sigma}{\tau^{*}})-1\right\}\| \eta\|^{2}-\beta e^{-\sigma}\|\eta^{\tau^{*}}\|^{2}\\\\
& &+2\|\eta\|\| P\|\|D(\varepsilon)(f(\hat{x},x^{\tau^{*}},u,u^{\tau^{*}})-f(x,x^{\tau(t)},u,u^{\tau(t)}))\|.\end{array}$$
Now, the majorization of the term $\|D(\varepsilon)(f(\hat{x},\hat{x}^{\tau^{*}},u,u^{\tau^{*}})-f(x,x^{\tau(t)},u,u^{\tau(t)}))\|.$
Characterizes the difference between the term that depends on the upper bound of the unknown delay and the term which depends on the unknown delay.
$$\begin{array}{lll}\|D(\varepsilon)(f(\hat{x},\hat{x}^{\tau^{*}},u,u^{\tau^{*}})-f(x,x^{\tau(t)},u,u^{\tau(t)}))\|&\leq&
\|D(\varepsilon)(f(\hat{x},\hat{x}^{\tau^{*}},u,u^{\tau^{*}})-f(x,x^{\tau^{*}},u,u^{\tau^{*}}))\|\\\\
&+ &\|D(\varepsilon)(f(x,x^{\tau^{*}},u,u^{\tau^{*}})-f(x,x^{\tau(t)},u,u^{\tau(t)}))\|
.\end{array}$$
Using \eqref{dd} and assumption $\textbf{A1.}$, on the one hand, we get
$$\begin{array}{lll}
\|D(\varepsilon)(f(\hat{x},\hat{x}^{\tau^{*}},u,u^{\tau^{*}})-f(x,x^{\tau^{*}},u,u^{\tau^{*}}))\|&\leq&
\gamma_{1}(\varepsilon)\displaystyle\sum_{i=1}^{n}\varepsilon^{i-1}|\hat{x}_{i}-x_{i}|+\gamma_{2}(\varepsilon)
\displaystyle\sum_{i=1}^{n}\varepsilon^{i-1}|\hat{x}_{i}^{\tau^{*}}-x_{i}^{\tau^{*}}|\\
&\leq &n\gamma_{1}(\varepsilon) \| D(\varepsilon)e\|
+n\gamma_{2}(\varepsilon) \| D(\varepsilon)e^{\tau^{*}}\|.\end{array}$$
Thus \begin{equation}\|D(\varepsilon)(f(\hat{x},\hat{x}^{\tau^{*}},u,u^{\tau^{*}})-f(x,x^{\tau^{*}},u,u^{\tau^{*}}))\| \leq
n\gamma_{1}(\varepsilon) \| \eta\| +n\gamma_{2}(\varepsilon) \| \eta^{\tau^{*}}\|, \label{major}
\end{equation}
and on the other hand,
\begin{equation}
\|D(\varepsilon)(f(x,x^{\tau^{*}},u,u^{\tau^{*}})-f(x,x^{\tau(t)},u,u^{\tau(t)}))\|
\leq \gamma_{2}(\varepsilon)\displaystyle\sum_{i=1}^{n}\varepsilon^{i-1}|x^{\tau^{*}}_{i}-x_{i}^{\tau(t)}|+\gamma_{3}(\varepsilon)
\displaystyle\sum_{i=1}^{n}\varepsilon^{i-1}|u_{i}^{\tau^{*}}-u_{i}^{\tau(t)}|.\label{a2}\end{equation}
From assumption $\textbf{H1.}$, there exists a bounded constant $\nu_{1}$ and $\nu_{2}$ such that \eqref{a2} can be written as follows:
\begin{equation*}
\|D(\varepsilon)(f(x,x^{\tau^{*}},u,u^{\tau^{*}})-f(x,x^{\tau(t)},u,u^{\tau(t)}))\|
\leq \gamma_{2}(\varepsilon)\frac{1-\varepsilon^{n}}{1-\varepsilon}\nu_{1}+\gamma_{3}(\varepsilon)
\frac{1-\varepsilon^{n}}{1-\varepsilon}\nu_{2}.\end{equation*}
where $\nu_{1}$ and $\nu_{2}$ are respectively, the positive constant which refers to the boundedness of $\|x^{\tau^{*}}-x^{\tau(t)}\|$ and
 $\|u^{\tau^{*}}-u^{\tau(t)}\|$. Let
$$\theta=\left\{
 \begin{array}{ll}
2\|P\|(\gamma_{2}(\varepsilon)+\gamma_{3}(\varepsilon))\frac{1-\varepsilon^{n}}{1-\varepsilon}\nu, \ \ \ \ if\ \varepsilon\neq1& \hbox{;} \\
2\|P\|(\gamma_{2}(\varepsilon)+\gamma_{3}(\varepsilon))n\nu, \ \ \ \ \ \ \ \ \ if\ \varepsilon=1& \hbox{.}
\end{array}
 \right.$$where $\nu=\max(\nu_{1},\nu_{2})$.
Thus \begin{equation}
2\|P\|\|D(\varepsilon)(f(x,x^{\tau^{*}},u,u^{\tau^{*}})-f(x,x^{\tau(t)},u,u^{\tau(t)}))\|
\leq \theta.\label{a3}\end{equation}
Using \eqref{major} and \eqref{a3}, we have  $$\begin{array}{lll}
\dot{W} (t,\eta_{t})+\frac{\sigma}{\tau^{*}}W (t,\eta_{t})
&\leq& -\left\{\lambda_{\min}(P)(\frac{1}{\varepsilon\|P\|}-\frac{\sigma}{\tau^{*}})-2 n\gamma_{1}(\varepsilon) \| P\| -1\right\}\| \eta\|^{2}\\\\
& &+\|\eta\|\theta+2n\gamma_{2}(\varepsilon)\| P\|  \| \eta \|  \| \eta^{\tau^{*}}\|- \beta e^{-\sigma }\|\eta^{\tau^{*}}\|^{2}
\end{array}$$
Using the fact that $$ \theta\|\eta\|\leq \frac{1}{4}\|\eta\|^{2}+\theta^{2}$$
we deduce that
\begin{equation}
\dot{W} (t,\eta_{t})+\frac{\sigma}{\tau^{*}}W (t,\eta_{t})-\theta^{2}
\leq -c(\sigma,\varepsilon) \|\eta\|^{2}+
b(\varepsilon)\|\eta \| \| \eta^{\tau^{*}}\|- \beta e^{-\sigma }\|\eta^{\tau^{*}}\|^{2}\label{w+sigw}
\end{equation}
with
$$\begin{array}{lll}c(\sigma,\varepsilon)&=&\lambda_{\min}(P)(\frac{1}{\varepsilon\|P\|}-\frac{\sigma}{\tau^{*}})-
2 n\gamma_{1}(\varepsilon) \| P\|-\frac{5}{4},\\\\
\ \ b(\varepsilon)&=&2n\gamma_{2}(\varepsilon)\| P\|. \end{array}$$
Now, the right side of the inequality \eqref{w+sigw} can be rewritten as follows
$$-c(\sigma,\varepsilon) \|\eta\|^{2}+
b(\varepsilon)\|\eta \|  \| \eta^{\tau^{*}}\|- \beta e^{-\sigma }\|\eta^{\tau^{*}}\|^{2}
=-\left(c(\sigma,\varepsilon)-\frac{b^{2}(\varepsilon)}{4\beta}e^{\sigma}\right)
\|\eta\|^{2}-\left(\frac{b(\varepsilon)}{2\sqrt{\beta}}e^{\frac{\sigma}{2} }\|\eta\|-\sqrt{\beta}e^{-\frac{\sigma}{2} }\|\eta^{\tau^{*}}\|\right)^{2}.$$
To satisfy inequality \eqref{w+sigw}, all we need to do is to choose $\sigma$ such that
$$\left(c(\sigma,\varepsilon)-\frac{b^{2}(\varepsilon)}{4\beta}e^{\sigma}\right)>0,$$
which is equivalent to
\begin{equation} \lambda_{\min}(P)\frac{\sigma}{\tau^{*}}+\frac{b^{2}(\varepsilon)}{4}(e^{\sigma}-1)<\frac{\lambda_{\min}(P)}{\varepsilon\|P\|}
-2 n\gamma_{1}(\varepsilon) \| P\| -\frac{5}{4}-\frac{b^{2}(\varepsilon)}{4\beta}\label{sigma2}\end{equation}
As in the proof of Theorem \ref{theorem1} and
let $\sigma \in]0, \delta[,$ then inequality \eqref{sigma2} is verified.

Now, the objective is to prove the uniform practical stability of \eqref{erro2}.\\
Using \eqref{condi2}, inequality \eqref{w+sigw} becomes
$$\dot{W}(\eta_{t})\leq-\frac{\sigma}{\tau^{*}} W(\eta_{t})+\theta^{2}$$
It follows that\begin{equation}W(\eta_{t})\leq e^{-\frac{\sigma}{\tau^{*}}t} W(\eta_{t}(0))+2\theta^{2}\frac{\tau^{*}}{\sigma}\label{gronw2}\end{equation}
Using  \eqref{norm}, we have, on the one hand,
$$\begin{array}{lll}
W (\eta_{t}(0))&\leq& \lambda_{\max}(P)\|\eta_{t}(0)\|^{2}+\displaystyle\int_{-\tau^{*}}^{0}e^{\frac{\sigma}{\tau}s}\|\eta(s)\|^{2}ds\\\\
&\leq &(\lambda_{\max}(P)+\tau^{*})\displaystyle\sup_{s\in[-\tau,0]}\|\eta(s)\|^{2}
\end{array}$$
and on the other hand,$$\lambda_{\min}(P)\|\eta_{t}\|^{2}\leq V (\eta_{t}).$$
We deduce that $$\|\eta_{t}\|\leq \sqrt{\frac{\|P\|+\tau^{*}}{\lambda_{\min}(P)}}e^{\frac{-\sigma}{2\tau^{*}}t}
\displaystyle\sup_{s\in[-\tau^{*},0]}\|\eta(s)\|+\sqrt{\frac{2\theta^{2}\tau^{*}}{\sigma\lambda_{\min(P)}}}.$$
Finally, with $\eta = D(\varepsilon)e$, the observation error $e(t)$ is given by
$$\|e(t)\|\leq \frac{1}{\|D(\varepsilon)\|}\sqrt{\frac{\|P\|
+\tau^{*}}{\lambda_{\min}(P)}}e^{\frac{-\sigma}{2\tau}t}\displaystyle\sup_{s\in[-\tau,0]}\|\eta(s)\|
+\frac{1}{\|D(\varepsilon)\|}\sqrt{\frac{2\theta^{2}\tau^{*}}{\sigma\lambda_{\min(P)}}}.$$
Then, the error dynamics \eqref{erro2} is globally (on K) practically exponentially stable.\end{proof}
\begin{remark}
It is noted that  if system \eqref{3} satisfies condition \eqref{lips} inspired in \cite{ibrir} and \cite{ghanes},
 then assumption $\textbf{A1.}$ is satisfied with
$\gamma_{1}(\varepsilon) = \gamma_{2}(\varepsilon) = \gamma_{3}(\varepsilon) = k(1 +\varepsilon + \cdots + \varepsilon^{n-1}).$ Let
$$\begin{array}{lll}
c_{1}(\varepsilon)&=&\frac{\lambda_{\min}(P)}{\varepsilon\|P\|}
-2 n\gamma_{1}(\varepsilon) \| P\| -n^{2}\gamma_{2}^{2}(\varepsilon)\| P\|^{2}-1,\\\\
c_{2}(\varepsilon)&=&\frac{\lambda_{\min}(P)}{\varepsilon\|P\|}
-2 n\gamma_{1}(\varepsilon) \| P\| -n^{2}\gamma_{2}^{2}(\varepsilon)\| P\|^{2}-\frac{5}{4}.
\end{array}$$
In this case, $c_{1}(\varepsilon)$ and $c_{2}(\varepsilon)$
tend to $\infty$ as $\varepsilon$ tends to zero. This implies that there exists $\varepsilon_{1} >0$
such that for all $0<\varepsilon <\varepsilon_{1}$ conditions \eqref{condi1} and \eqref{condi2} are fulfilled.
\end{remark}
\section{Numerical example}
 This section presents experimental result, we give an example of the orientational motion of polar molecules acted on by an external perturbation.
 We consider a physical model corresponding to a slow relaxation process. The dynamics model systems are represented by:
\begin{equation}\begin{array}{lll}
\dot{x}_{1}&=&x_{2}+\frac{1}{12}\sin x_{2}(t-\tau(t))+\frac{1}{24}x_{1}\cos(u(t-\tau(t)),\\
\dot{x}_{2}&=&x_{3}+\frac{1}{12}x_{2}(t-\tau(t))+\frac{1}{24}x_{2}+u(t-\tau(t)),\\
y&=&x_{1}.
\end{array}\label{*}\end{equation}
where the input $u(t)=\cos(7t)$ denote the orientational potential energy, the function $\tau(t)$ is defined as follows: $\tau(t)=0.25+0.01\cos^{2}(t)$
being the Debye relaxation time $x(t)$ is the augmented state vector containing the plant state vector.
Following the notation used throughout the paper, let
$f_{1}(x,x^{\tau(t)},u,u^{\tau(t)})=\frac{1}{12}\sin x_{2}(t-\tau(t))+\frac{1}{24}x_{1}\cos(u(t-\tau(t)))$,
$f_{2}(x,x^{\tau(t)},u,u^{\tau(t)})=\frac{1}{12}x_{2}(t-\tau(t))+\frac{1}{24}x_{2}+u(t-\tau(t)$.
Furthermore, since $f_{1}$ depends on $x_{2}^{\tau(t)}$,  the method proposed in \cite{ghanes} is not applicable in this case.
The input and  the states are bounded, which make assumption H1 holds.
It is easy to check that system \eqref{*} satisfies Assumption A1
with $\gamma_{1}(\varepsilon) =\frac{1}{24}(1+\varepsilon),\, \gamma_{2}(\varepsilon) =\frac{1}{12}(1+\varepsilon)$
and $\gamma_{3}(\varepsilon) =(1+\varepsilon).$\\
The initial conditions for the system are $x(0) = \left[
            \begin{array}{ccc}
              -2& 1 \\
            \end{array}
          \right]^{T}$,
for the observer $\hat{x}(0) = \left[
            \begin{array}{ccc}
              2& 2 \\
            \end{array}
          \right]^{T}$.
Now, select $L=\left[
            \begin{array}{ccc}
              -5& -5 \\
            \end{array}
          \right]^{T}$ and $A_{L}$ is Hurwitz.  Using Matlab, the solutions of the Lyapunov
equations \eqref{Ly2} is given  by
$$ P=\left[
  \begin{array}{ccc}
     0.1200  &  0.1000\\
    0.1000  &  1.1000\\
  \end{array}\right].$$
So, $\lambda_{\max}(P)=1.1101$ and $\lambda_{\min}(P)=0.1099$.  This implies that condition \eqref{condi2} is satisfied for all $0<\varepsilon<0.06$.
According to the practical stability improved in the proof of Theorem \ref{th2}, it is clear from Figs. 1 and 2
that  the global uniform practical convergence is ensured with radius
For our numerical simulation, we choose $\varepsilon=0.05.$

\section{Conclusion}
In this paper, the problem of global uniform practical exponential stability and  an observer for a class of time-delay nonlinear
systems have been considered.
This class of systems covers the systems having a triangular structure. In the case of a constant time delay,
we have derived delay-independent conditions to ensure global exponential stability.
We have suggested sufficient conditions to guarantee a practical stability of the proposed observer
in the case of a bounded and unknown variable time delay.
Finally, the effectiveness of the conditions
obtained in this paper is verified in a numerical example.

 \end{document}